\documentclass{amsart}
\usepackage{amssymb}
\usepackage{amsmath}

\usepackage{a4}
\usepackage{color}

\begin{document}
%\date{\version}
\newtheorem{theorem}{Theorem}[section]
\newtheorem{lemma}[theorem]{Lemma}
\newtheorem{proposition}[theorem]{Proposition}
\newtheorem{remark}[theorem]{Remark}
\newtheorem{definition}[theorem]{Definition}
\newtheorem{corollary}[theorem]{Corollary}
\newtheorem{example}[theorem]{Example}
\def\qedbox{\hbox{$\rlap{$\sqcap$}\sqcup$}}
\makeatletter
  \renewcommand{\theequation}{%
   \thesection.\alph{equation}}
  \@addtoreset{equation}{section}
\makeatother
\title[Structure groups of model spaces]
{On the structure group of a decomposable model space}
%%%%%
\author{Corey Dunn*}
\begin{address}{CD: Mathematics Department, California State University at San Bernardino,
San Bernardino, CA 92407, USA. Email: \it
cmdunn@csusb.edu.}\end{address}
%%%%%
\author{Cole Franks}
\begin{address}{CF: Mathematics Department, University of South Carolina,
Columbia, SC 29225, USA. Email: \it
franksw@email.sc.edu.}\end{address}
%%%%%
\author{Joseph Palmer}
\begin{address}{JP:  Mathematics Department, Washington University in St. Louis, St. Louis, MO 63130, USA. Email: \it
jpalmer@math.wustl.edu.}\end{address}

\begin{abstract}

We study the structure group of a canonical algebraic curvature tensor built from a symmetric bilinear form, and show that in most cases it coincides with the isometry group of the symmetric form from which it is built.  Our main result is that  the structure group of the direct sum of such canonical algebraic curvature tensors on a decomposable model space must permute the subspaces $V_i$ on which they are defined.  For such an algebraic curvature tensor, we show that if the vector space $V$ is a direct sum of subspaces $V_1$ and $V_2$, the corresponding structure group decomposes as well if $V_1$ and $V_2$ are invariant of the action of the structure group on $V$.  We determine the freedom one has in permuting these subspaces, and  show these subspaces are invariant if $\dim V_1 \neq \dim V_2$ or if the corresponding symmetric forms defined on those subspaces have different (but not reversed) signatures, so that in this situation, only the trivial permutation is allowable.  We  exhibit a model space that realizes the full permutation group, and, with exception to the balanced signature case,  show the corresponding structure group is isomorphic to the wreath product of the structure group of a given symmetric bilinear form by the symmetric group.  Using these results, we conclude that the structure group of any member of this family is isomorphic to a direct product of wreath products of pseudo-orthogonal groups by certain subgroups of the symmetric group.  Finally, we apply our results to two families of manifolds to generate new isometry invariants that are not of Weyl type.

\end{abstract}

\keywords{decomposable model space, structure group, canonical algebraic curvature tensor, curvature homogeneous. \newline
*corresponding author \newline
} \maketitle

\section{Introduction} 

Let $V$ be a real vector space of finite dimension $N$, let $V^* :={\rm Hom}(V, \mathbb{R})$ be its dual.  An object $R \in \otimes^4V^*$ is called an \emph{algebraic curvature tensor} if it satisfies the following three properties, the last of which is known as the \emph{Bianchi identity}:
\begin{equation} \label{R}
 \begin{array}{r c l}
R(x, y, z, w)& = & - R(y, x, z, w), \\
R(x, y, z, w) & = & R(z, w, x, y), {\rm and } \\
0 & = & R(x, y, z, w) + R(x, z, w, y) \\
 & &\quad\qquad+ R(x, w, y, z)\,.
\end{array}
\end{equation}
     
If $(M, g)$ is a pseudo-Riemannian manifold, then one may use the Levi-Civita connection $\nabla$ to compute the Riemann curvature tensor $R^\nabla \in \otimes^4 T^*M$, and the evaluation of this tensor at a point $P\in M$ produces the algebraic curvature tensor $R^\nabla_P \in \otimes^4 T^*_PM$, where $T^*M$ is the cotangent bundle of $M$, and $T^*_PM$ is the cotangent space of $M$ at $P$.  It is a classical differential geometric fact that every algebraic curvature tensor $R$ can be realized as the curvature tensor of a pseudo-Riemannian manifold at a point \cite{G01}.   Thus it can be said that these algebraic curvature tensors are an algebraic portrait  of the curvature of a manifold at a point, and an understanding of these algebraic objects often translates into a subsequent understanding of the geometrical object they represent.   For example, an understanding of the Osserman conjecture in the higher signature setting is concerned with an algebraic understanding of the Jordan normal form of the Jacobi operator \cite{G01}, Stanilov-Tsankov theory \cite{BFG07} is concerned with the commutativity of certain other natural operators associated to the Riemann curvature tensor, and other authors have studied certain algebraic questions concerning these algebraic curvature tensors simply because these questions are of interest in their own right.  Several examples of this include work on the algebraic properties of the Jacobi operator on complex model spaces \cite{G07},  the study of the linear independence of certain sets of algebraic curvature tensors \cite{DD10}, and results aimed at improving the efficiency with which one may express a given algebraic curvature tensor \cite{DD10, DG04, K91}.

Let $\alpha_1, \ldots, \alpha_n$ be a collection of contravariant tensors on $V$.  We call the tuple $\mathfrak{M}:=(V, \alpha_1, \ldots, \alpha_n)$ a \emph{model space}.  For example, if $\varphi$ is a symmetric bilinear form, and $R$ is an algebraic curvature tensor, $(V, \varphi, R)$ is a model space.  There are some places in the literature where it has been convenient to distinguish certain types of model spaces from others.  For example, in \cite{D09, G07} the pair $(V, R)$ is referred to as a \emph{weak model space}, although in the current work it is not necessary to make this distinction.

There is a natural action of the general linear group $Gl(V)$ on any contravariant tensor $\alpha \in \otimes^sV^*$.  Namely, if $A \in Gl(V)$, we may define $$(A^* \alpha)(x_1, \ldots, x_s)= \alpha(Ax_1, \ldots, Ax_k).$$  This differs slightly from the standard representation theoretic approach, since one would normally need to define $A^*\alpha$ by first precomposing with $A^{-1}$ to ensure that $\rho(A)(\alpha) = A^*\alpha$ defines a representation (i.e., that $\rho: Gl(V) \to {\rm End}(\otimes^s V^*)$ is a homomorphism).  We define the \emph{structure group} $G_{\mathfrak{M}}$ of a model space $\mathfrak{M} = (V, \alpha_1, \ldots, \alpha_n)$ as
$$
G_{\mathfrak{M}} = \{A \in Gl(V) | A^* \alpha_i = \alpha_i {\rm\ for\ } i = 1, \ldots, n\}.
$$
In the event that $n = 1$ so that the model space $\mathfrak{M} = (V, \alpha)$, then we sometimes write $G_{\mathfrak{M}} = G_{\alpha}$ for simplicity when there is no confusion as to what is meant.  In addition, we may also refer to $G_\alpha$ as the structure group of $\alpha$ for simplicity, rather than as the structure group of the model space $(V, \alpha)$.

Structure groups arise under different names in situations that are familiar to mathematicians.  For example, if $\varphi$ is a positive-definite inner product, then $G_{\varphi} = O(N)$, the familiar orthogonal group.  If one notes that $Gl(V)$ is the structure group of the trivial model space consisting solely of $V$, then many important quantities are dependent upon the observation that the quantity they compute be ``independent of the particular basis chosen,''  the determinant and trace of a linear operator, for example.  

There are many nontrivial examples where an understanding of a model space's structure group gives rise to more significant and useful information.  A pseudo-Riemannian manifold $(M,g)$ is called \emph{curvature homogeneous}  if there is a model space $\mathfrak{M}:= (V, \varphi, R)$ and for all $P \in M$ there exists a linear isometry $\phi_P: V \to T_PM$ with $\phi_P^*R_P^\nabla = R$, where $\varphi$ is an inner product with the same signature as $g$, and $R$ is an algebraic curvature tensor.  Since the models $\mathfrak{M}_P := (T_PM, g|_P, R|_P)$ are all isomorphic to $\mathfrak{M}$, the \emph{structure group} of a curvature homogeneous manifold is the structure group of any of the models $\mathfrak{M}_P$.  It is common to search for a non-constant isometry invariant to determine when a curvature homogeneous manifold is not locally homogeneous, one of the aims in the broad study of curvature homogeneity.  See, for example \cite{KTV92} in the Riemannian setting, and \cite{D00, DG05} in the higher signature setting.  The Weyl scalar invariants usually provide such an invariant, and in the Riemannian setting, if they do not, then the manifold is locally homogeneous \cite{PTV96}.  It is the case in the higher signature setting that all of these scalar invariants could vanish (Walker metrics are such a family \cite{BGGNV09, GS05}); in this case one needs to look further.  It is therefore a common practice to compute the structure group $G_{\mathfrak{M}}$ of the manifold in question, and produce a quantity (using a geometric quantity other than just $R$, for example, the covariant derivatives of $R$)  that is invariant under the action of this group.  Such a quantity $\alpha: \tilde{\mathfrak{M}} \to K$, where $K$ is an auxiliary space and $\tilde{\mathfrak{M}} := (V, \varphi, R, A_1, \ldots, A_k)$,  is called an \emph{invariant} of $\tilde{\mathfrak{M}}$, we give a formal definition below (see also page 26 of \cite{G07} for more details).   Using the structure group of a model space to construct invariants is common, some representative examples are \cite{DG05, DGS04, GS04}, and most recently in \cite{D09}.   Thus, an understanding of $G_{\mathfrak{M}}$ is useful in its own right, and crucial to the study of curvature homogeneity in the higher signature setting.

\begin{definition}  \label{invariant}
Let $\mathfrak{M} = (V, \alpha_1, \ldots, \alpha_n)$ and $\mathfrak{M}' = (V, \beta_1, \ldots, \beta_m)$ be model spaces with $\{\beta_1, \ldots, \beta_m\} \subseteq \{\alpha_1, \ldots, \alpha_n\}$.  The structure group $G_{\mathfrak{M}'}$ acts on the $\alpha_i$ by precomposition.  For some auxiliary space $K$, a function $\alpha: \mathfrak{M} \to K$ is an \emph{invariant} (or \emph{model space invariant}) if $\alpha$ is invariant under the action of $G_{\mathfrak{M}'}$ on $\mathfrak{M}$.
\end{definition}

\begin{remark}  \label{remark1}
{\rm Definition \ref{invariant} may seem a bit obscure, although we remind the reader that this is the natural setting for the invariants one typically encounters.  For example, if $\mathfrak{M} = (V, \varphi, R)$, where $\varphi \in S^2(V)$ is nondegenerate and $R\in \mathcal{A}(V)$, then one can consider an orthonormal basis $\{e_1, \ldots, e_N\}$ and define the \emph{Ricci tensor} $\rho$ and \emph{scalar curvature} $\tau$ as
$$
\begin{array} {r c l}
\rho(x,y) & = & \sum \varphi(e_i, e_i)R(x, e_i, e_i, y), {\rm\ and} \\
\tau & = & \sum \varphi(e_j, e_j)\rho(e_j, e_j) = \sum \varphi(e_i, e_i) \varphi(e_j, e_j) R(e_i, e_j, e_j, e_i).
\end{array}
$$
That $\rho$ and $\tau$ are independent of the orthonormal basis chosen is precisely the condition that $\rho$ and $\tau$ are invariant under the action of the structure group $G_{\mathfrak{M}'}$, where $\mathfrak{M}' = (V, \varphi)$.  In addition, if $\omega$ is any Weyl scalar invariant (i.e., an invariant of $(V, \varphi, R, \nabla R, \ldots, \nabla^n R)$ that arises from using the metric $\varphi$  to fully contract all indices using combinations of $R, \nabla R, \ldots, \nabla^n R$, see page 15 of \cite{G07}), then one may use an orthonormal basis to express $\omega$ in the components of the tensors involved, and then show that the expression is independent of the orthonormal basis chosen.  Any of these Weyl scalar invariants are model space invariants as in Definition \ref{invariant}, where again $\mathfrak{M}' = (V, \varphi)$.  
}\end{remark}

\begin{remark}
{\rm Although the Weyl scalar invariants are a large class of model space invariants, the auxiliary space $K$ need not always be the real numbers.  For example, $\rho$ is an $S^2(V)$-valued invariant, and there are several invariants in \cite{D09} which are valued in symmetric projective space.  This justifies the added generalization of allowing these model space invariants to be $K$-valued in Definition \ref{invariant}, rather than only real valued.
}\end{remark}

\begin{remark}
{\rm  The main use for these model space invariants in this paper is to construct isometry invariants in the context of pseudo-Riemannian manifolds.  See the discussion in the second paragraph of Section \ref{section5}.
}\end{remark}

It is the goal of this paper to compute the structure group of a large and useful family of model spaces, and to study the effect of the decomposition of certain model spaces on the corresponding structure group.  In addition, we illustrate the application of these results in two geometrical situations, and use our structure group results to construct useful model space invariants.  We make these goals more precise and formally state our results.

Let $S^2(V)$ be the space of symmetric bilinear forms on $V$, and let $\mathcal{A}(V)$ be the vector space of algebraic curvature tensors.  We define 
$$
R_\varphi(x, y, z, w) = \varphi(x, w)\varphi(y,z) - \varphi(x, z)\varphi(y,w).
$$
It is well known \cite{G01} that $R_\varphi \in \mathcal{A}(V)$, and that $\mathcal{A}(V) = {\rm Span}\{R_\varphi | \varphi \in S^2(V)\}.$  The $R_\varphi$ have geometrical significance as well: if $(M,g)$ is isometrically embedded in pseudo-Euclidean space $\mathbb{R}^{\dim(M) + \kappa}$, then the curvature tensor of $M$ is of the form $\sum_{i = 1}^\kappa \pm R_{\varphi_i}$ (see, for example \cite{DG04}).  For these reasons (in particular, when $\kappa = 1$), the tensor $R_\varphi$ is sometimes called a \emph{canonical} algebraic curvature tensor.

There is one final preliminary notion to introduce before we introduce our main results and begin our study.  Let $\mathfrak{M} = (V, \alpha_1, \ldots, \alpha_n)$ be a model space. Suppose $V = V_1 \oplus V_2$, where $\dim(V_i) \geq 1$.  If  $\alpha \in \otimes^tV^*$, then we write $V_1 \perp_\alpha V_2$ if for $x_i \in V_i$, we have $\alpha(x_1, x_2, v_1, \ldots, v_{t-2}) = \alpha(x_2, x_1, v_1, \ldots, v_{t-2}) = 0$ for all $v_1, \ldots, v_{t-2} \in V$.  We say that $\mathfrak{M}$ is \emph{decomposable} if there exist subspaces $V_1$ and $V_2$ with $V = V_1 \oplus V_2$, $\dim(V_i) \geq 1$, and $V_1 \perp_{\alpha_s} V_2$ for every $s = 1, \ldots, n$.  In this event, we write $\alpha_s = \alpha^1_s \oplus \alpha^2_s$, where $\alpha^i_s$ is the restriction of $\alpha_s$ to $V_i$, and $\mathfrak{M} = \mathfrak{M}_1 \oplus \mathfrak{M}_2,$ where $\mathfrak{M}_i = (V_i, \alpha^i_1, \ldots, \alpha^i_n)$.  We say that $\mathfrak{M}$ is \emph{indecomposable} if it is not decomposable.  For example $R \in \mathcal{A}(V)$, and 

\begin{equation}  \label{kerR}
\ker (R) := \{ v\in V | R(v, x, y, z) = 0 {\rm\ for\ all\ } x, y, z \in V\} \neq 0,
\end{equation} 
then there is the decomposition $(V, R) \cong (\bar V, \bar R) \oplus (\ker(R), 0)$, where $\pi: V \to V / \ker (R) = \bar V$, and $\bar R$ is characterized by $\pi^* \bar R = R$.  It follows that $\ker(\bar R) = 0$.  We will investigate this situation in Section \ref{section2}.  We remark that our definition of indecomposable above is biased toward the first two slots of the tensor, although in the case of algebraic curvature tensors and symmetric bilinear forms there is no bias.  In addition, it seems that the definition of $\ker(R)$ above is also biased toward the first slot, and we shall see that this is not the case: one would define the same object by specifying any of the other slots as well.  See Propositions \ref{Rbias} and \ref{kerbias} in Section \ref{section2} for a verification of these facts.  

The concept of indecomposability is also a familiar one:   consider again the  situation that $\varphi \in S^2(V)$ is positive definite.  One usually refers to a decomposition $(V, \varphi) = (V_1, \varphi_1) \oplus (V_2, \varphi_2)$ simply as $V_1 \perp V_2$.  Notice that in the associated structure group $G_\varphi = O(N)$, the subspaces $V_1$ and $V_2$ are never $A-$invariant for all $A \in G_{\varphi}$.  Also note that $O(N)$ does not decompose as the group theoretic internal direct product $O(\dim(V_1)) \times O(\dim(V_2))$.  In fact, since orthonormal bases exist on inner product spaces, one always has the complete decomposition $(V, \varphi) = \oplus_{i = 1}^N (V_i, \varphi_i)$, where $\dim(V_i) = 1$.  Our main result, Theorem \ref{multiphi}, offers a very different picture of this situation using the model space $\oplus_{i = 1}^k(V_i, R_{\varphi_i})$, where $\varphi_i \in S^2(V)$.  See Theorem \ref{multiphi} and Section \ref{section4} for a description of the effect of this decomposition on its corresponding structure group.

We now give an outline of this paper and list our main results.    Section \ref{section2} is a short survey of prerequisite information, and gives a general characterization of $G_R$ in the event $\ker(R) \neq \{0\}$, when $R \in \mathcal{A}(V)$.  We then give a computation of the structure group $G_{R_\varphi}$ when $\varphi$ is nondegenerate in Theorem \ref{Gphi}, completing the characterization of $G_{R_{\varphi}}$ in terms of $G_{\varphi}$.  Assertions (3) and (4) of Theorem \ref{Gphi} are not new results, since Assertion (3) is observational (although see \cite{DGS04} for a nontrivial application), while Assertion (4) is obvious upon consideration of the curvature identities in Equation (\ref{R});  we include them here for completeness.

\begin{theorem}  \label{Gphi}
Suppose $\varphi \in S^2(V)$ is nondegenerate and has signature $(p,q)$.
\begin{enumerate}
\item  If ${\rm Rank}(\varphi) \geq 3$, then $G_{R_\varphi} = \{A \in Gl(V) | A^* \varphi = \pm \varphi\}$.
\item  If ${\rm Rank} (\varphi) \geq 3$, and $p \neq q$, then $G_{R_{\varphi}} = G_{\varphi}$.
\item  If ${\rm Rank}(\varphi) = 2$, then $G_{R_\varphi} = \{A \in Gl(V) | \det(A) = \pm 1\}$.
\item  If ${\rm Rank}(\varphi) \leq 1$, then $R_{\varphi} = 0$, so $G_{R_{\varphi}} = Gl(V)$.
\end{enumerate}
\end{theorem}

Section \ref{section3} is an investigation of $G_R$, when $R = \oplus_{i = 1}^k R_{\varphi_i}$.  Such an $R$ has geometric significance as the curvature tensor of a certain hypersurface embedding, or as a skew-Tsankov curvature tensor on a Riemannian model space \cite{BFG07}; we consider this sort of tensor in the geometric setting in Section \ref{section5}.  In Section \ref{section3}, we lay out the proof of our main result Theorem \ref{multiphi} in the technical Lemma \ref{technical}; with exception to Corollary \ref{directproduct}, the conclusion of Theorem \ref{multiphi} will be the foundation of our subsequent results in this paper.  Denote the group of permutations of $\{1, \ldots, k\}$ as $S_k$.

\begin{theorem}  \label{multiphi}
Let $\varphi_i \in S^2(V)$.  Suppose $(V, R) = \oplus_{i = 1}^k (V_i, R_{\varphi_i})$ is a model space with $\ker R = 0$, and let $A \in G_R$.  Then there exists  $\sigma\in S_k$ so that $A:V_i \to V_{\sigma(i)}$.  
\end{theorem}

Section \ref{section4} is devoted to the corollaries of Theorem \ref{multiphi}, and we give a complete description of $G_R$ for $R = \oplus_{i = 1}^kR_{\varphi_i}$ in terms of the signatures of the defining forms $\varphi_i$.     Corollary \ref{directproduct} demonstrates that if the model space decomposes into subspaces which are invariant under the action of the structure group, then the structure group itself decomposes as a  (group theoretic) internal direct product.

\begin{corollary}  \label{directproduct}
 Let $(V, R) = (V_1, R_1) \oplus (V_2, R_2)$, and suppose that $V_1$ and $V_2$ are $g-$invariant for all $g \in G_R$.  Then $G_R \cong G_{R_1} \times G_{R_2}$, the group theoretic internal direct product of $G_{R_1}$ and $G_{R_2}$.
\end{corollary}

 We use Corollary \ref{directproduct} and Theorem \ref{multiphi} to show in Corollary \ref{dimensions} that it is impossible for any member of the structure group to permute subspaces of different dimension.  Corollary \ref{dimensions} will also demonstrate that unless the signature of the forms involved are compatible, it is also impossible to permute subspaces of the same dimension.  According to Theorem \ref{multiphi}, there is a well-defined subgroup $S_k^R \leq S_k$ corresponding to every algebraic curvature tensor $R$ of the form found in Theorem \ref{multiphi}.  The following corollary demonstrates that this subgroup $S_k^R \neq S_k$ should the dimension of any of the subspaces $V_i$ or, up to replacing $\varphi$ with $-\varphi$, the signature of the forms involved differ.   
 
\begin{corollary}  \label{dimensions}
Let $\varphi_s \in S^2(V_s)$ be nondegenerate for $s = 1, \ldots, k$, and suppose $(V, R) = \oplus_{s = 1}^k (V_s, R_{\varphi_s})$.   Let $A \in G_R$, and suppose that for some $i, j$,  the signature of $\varphi_s$ is $(p_s, q_s)$ for $s = i, j.$
\begin{enumerate}
\item  If $A: V_i \to V_j$, then $\dim(V_i) = \dim(V_j)$.
\item  If $A: V_i \to V_j$, and ${\rm Rank}(\varphi_i) \geq 2$, then $(p_i, q_i) =  (p_j, q_j)$ or $(q_j, p_j)$.
\item  If $(p_i, q_i) = (p_j, q_j)$ or $(q_j, p_j)$, then there exists a $B \in G_R$ where $B|_{V_i}: V_i \to V_j$.  In this event, we must have $(B|_{V_i})^* R_{\varphi_j} = R_{\varphi_i}$.
\end{enumerate}
\end{corollary}

We conclude Section \ref{section4} with a study of how freely the structure group may permute the subspaces $V_i$, again, in the situation of Theorem \ref{multiphi}.  Since Corollary \ref{dimensions} forbids the exchange of subspaces in certain situations, we consider the situation $V_i \cong V_j = W$, and $\varphi_i = \pm\varphi_j$ for every $i$ and $j$.  Without any loss of generality, we may freely replace $\varphi_j$ with $-\varphi_j$ if necessary and assume below in Corollary \ref{wreath} that $\varphi = \varphi_j$ for all $j$.

\begin{corollary} \label{wreath}
Suppose $(V, R) = \oplus_{i = 1}^k (W, R_{\varphi})$, where $\varphi\in S^2(W)$ is nondegenerate.  Then $G_R$ is isomorphic to the wreath product $G_{R_\varphi} \wr S_k$.
\end{corollary}

When Corollary \ref{wreath} is used in concert with Corollary \ref{directproduct}, we arrive at the following result, which easily generalizes from two direct summands $V = \oplus_{p = 1}^2 V_p$ to any finite number of direct summands, and completes the classification of the structure group of the algebraic curvature tensors considered here. 

\begin{corollary}  \label{morewreath}
Let $(V_p, R_p) = \oplus_{i = 1}^{k_p} (W_p, R_{\varphi_p})$, and let $(V, R) = (V_1, R_1) \oplus (V_2, R_2)$.  If $\dim W_1 \neq \dim W_2$, then $G_R \cong (G_{R_{\varphi_1}} \wr Sym_{k_1}) \times (G_{R_{\varphi_2}} \wr Sym_{k_2})$.
\end{corollary}

In Section 5 we prove Theorem \ref{invariants}:  

\begin{theorem} \label{invariants}
Let $\mathfrak{M} = \oplus^s_{i=1} (V_i, R_{\varphi_i}, A^i_1, \ldots, A^i_k)$ be a decomposible model space, and let $\alpha_i$ be an invariant of the model $\mathfrak{M}_i = (V_i,  R_{\varphi_i}, A^i_1, \ldots, A^i_k)$.  Then any symmetric function of the $\alpha_i$ is an invariant of $\mathfrak{M}$.
\end{theorem}

We give some introductory remarks before proving Theorem \ref{invariants} that put the theory of model space invariants into context, then we continue by determining the structure group of any Riemannian Skew-Tsankov manifold (see Definiton \ref{skewtsankov} and Theorem \ref{skewmodel}), and construct an isometry invariant that is not of Weyl type on a curvature homogeneous manifold modeled on the situation in Corollary \ref{morewreath} (see Theorem \ref{ppexample}).

We include a short and nontechnical summary of our results in Section \ref{section6} that summarizes our work, and gives a method of characterizing any structure group of any algebraic curvature tensor which is the direct sum of any finite number of canonical algebraic curvature tensors.

\section{Preliminary notions and the determination of $G_{R_{\varphi}}$}  \label{section2}

There are several preliminary comments we will need to make that will properly set the stage for our subsequent study of structure groups in general.  These preliminary comments are either observational, straightforward, or can be found in \cite{G07}.  We conclude this section with a computation of the structure group $G_{R_\varphi}$.

We recall that the definition of a decomposable model space seemed to favor the first two slots of the tensors involved, and that the definition of $\ker(R)$ in Equation (\ref{kerR}) seemed to favor the first entry of $R$.  We verify below that for the model spaces we consider in this paper, there is no concern.

\begin{proposition}  \label{Rbias}
Let $R \in \mathcal{A}(V)$, and suppose that $(V, R)= (V_1, R_1) \oplus (V_2, R_2)$ is decomposible.  Then for every $x_i \in V_i,$ and every $y, z \in V$, we have 
$$
R(x_1, y, x_2, z) = R(x_1, y, z, x_2) = R(y, x_1, x_2, z) = $$  $$R(y, x_1, z, x_2) = R(y, z, x_1, x_2) = 0.
$$
\end{proposition}
In other words, Proposition \ref{Rbias} demonstrates that within the realm of decomposible model spaces of the form $(V, R)$ where $R \in \mathcal{A}(V)$, the definition of decomposible is not biased toward the first two entries.  In addition, this Proposition applies equally well to any model space of the form $(V, \varphi, R)$ where $\varphi \in S^2(V)$ and $R \in \mathcal{A}(V)$.

\begin{proof}[Proof of Proposition \ref{Rbias}.]  Suppose that $(V, R) = (V_1, R_1) \oplus (V_2, R_2)$.  By using the symmetries in Equation (\ref{R}), one only needs to show $R(x_1, y, z, x_2) = 0$.  Decompose $y = y_1 + y_2,$ where $y_i \in V_i$.  Then by assumption, $R(x_1, y, z, x_2) = R(x_1, y_1, z, x_2)$.  By again using the symmetries in Equation (\ref{R}), using our hypothesis of decomposition, and decomposing $z = z_1 + z_2$ where $z_i \in V_i$, we have
$$
R(x_1, y_1, z, x_2) = R(z, x_2, x_1, y_1) = R(z_2, x_2, x_1, y_1) = R(x_1, y_1, z_2, x_2).
$$
We now use the Bianchi identity and again our hypothesis to show that $R(x_1, y, z, x_2) = R(x_1, y_1, z_2, x_2) = 0$ and complete the proof:
$$
\begin{array}{r c c c c c c}
0 & = & R(x_1, y_1, z_2, x_2) &+& R(x_1, x_2, y_1, z_2) &+& R(x_1, z_2, x_2, y_1) \\
& = & R(x_1, y_1, z_2, x_2) &+& 0 &+& 0 \\
& = & R(x_1, y_1, z_2, x_2).& & & & \end{array}$$
\end{proof}

\begin{proposition}  \label{kerbias}
Let $R \in \mathcal{A}(V)$, and let $\ker(R)$ be defined as in Equation (\ref{kerR}).  We have
$$
\begin{array}{r c l}
\ker(R) & = & \{ v \in V | R(x, v, y, z) = 0 {\rm\ for\ all\ } x, y, z \in V\} \\
& = & \{ v \in V | R(x, y, v,  z) = 0 {\rm\ for\ all\ } x, y, z \in V\} \\
& = & \{ v \in V | R(x, y, z, v) = 0 {\rm\ for\ all\ } x, y, z \in V\}.
\end{array}
$$
\end{proposition}
\begin{proof}
Let $v\in V$ be given, and $x, y,$ and $z$ arbitrary.  Then using the symmetries in Equation (\ref{R}), we have $R(v, x, y, z) = -R(x, v, y, z) = R(y, z, v, x) = -R(y, z, x, v)$.
\end{proof}
We briefly remark that a similar proof and the second Bianchi identity shows that, if $\nabla R$ is an algebraic covariant derivative curvature tensor, one may define $\ker \nabla R$ to be the set of vectors that force $\nabla R$ to vanish, regardless of the slot one considers, in analogy to Proposition \ref{kerbias}.  See Assertion 1(c) of Theorem \ref{ppalreadyknown}.

Suppose $(V, R)$ is a model space with $R \in \mathcal{A}(V)$.  If $x \in \ker R$, and $A \in G_R$, then for $A\tilde y = y$, $A\tilde z = z,$ and $A\tilde w = w$, we have 
$$
R(Ax, \tilde y, \tilde z, \tilde w) = A^*R(x, y, z, w) = 0.
$$
Thus $A:\ker R \to \ker R$, and so $\ker R$ is an invariant subspace.  So choosing a basis for $\ker R$ and extending it to a full basis for $V$ demonstrates that every $A\in G_R$ takes the block matrix form with respect to any such basis
$$
A = \left[
\begin{array}{c c}
\bar A& 0 \\ C & N
\end{array}
\right],
$$
where $N \in Gl(\ker R)$, $C$ is \emph{any} submatrix of appropriate size, and $\bar A \in G_{\bar R}$, where $\bar R\in \mathcal{A}(\bar V)$, and $\bar V:=V/\ker R$, and $\bar R$ is the pullback of $R$ under the canonical projection $\pi: V \to \bar V$ as discussed in the introduction.  It is for this reason that we study model spaces $(V, R)$ where $\ker R = 0$, since if $\ker R \neq 0$, the computation of $G_R$ translates directly into a study of an associated algebraic curvature tensor with trivial kernal, as in Theorem \ref{Gphi}.  It is known \cite{G07} that if ${\rm Rank}(\varphi) \geq 2$, then $\ker R_{\varphi} = \ker \varphi$, so the assumption of nondegeneracy of $\varphi$ in Assertions (1)--(3) of Theorem \ref{Gphi} is equivalent to $R_\varphi$ having a trivial kernal.  

\begin{proof}[Proof of Theorem \ref{Gphi}.]
Suppose $A \in G_{R_\varphi}.$    By the definition of $R_\varphi$, we have $R_\varphi = A^*R_\varphi = R_{A^*\varphi}$.  If ${\rm Rank}(\varphi) \geq 3$, we may use Lemma 1.6.3 of \cite{G07} to conclude that $A^*\varphi = \pm \varphi$.   Since $R_\varphi = R_{-\varphi}$, the opposite containment holds, and Assertion (1) is established.  To prove Assertion (2), we must eliminate the possibility that $A^*\varphi = -\varphi$, i.e., $A$ is a para-isometry.  But it is well known that para-isometries  only exist in the balanced signature setting.  Indeed, such a map would exchange the causal type of any element in an orthonormal basis. 
If ${\rm Rank}(\varphi) = \dim V = 2$, then using the curvature symmetries, one has for $A \in G_{R_\varphi}$ that $R(x, y, y, x)$ is the only nonzero entry up to the symmetries found in Equation (\ref{R}), where $\{x, y\}$ is a  linearly independent set.  One easily computes
$$
R_{\varphi}(Ax, Ay, Ay, Ax) = (\det A)^2 R_{\varphi}(x, y, y, x).
$$
Since ${\rm Rank}(\varphi) = 2$, we have $R_{\varphi} \neq 0$, so $(\det A)^2 = 1$.  Assertion (3) now follows.
To prove Assertion (4), we note that $\mathcal{A}(V) = 0$ for any vector space of dimension less than 2.
\end{proof}

\section{The determination of $G_R$, where $R = \oplus R_{\varphi_i}$}  \label{section3}

Now that we have a complete understanding of $G_{R_{\varphi}}$, we turn our attention to pursuing an understanding of $G_R$ where $R = \oplus_{i = 1}^k R_{\varphi_i}$, on a model space that is decomposable.  According to the discussion preceding the proof of Theorem \ref{Gphi}, we need only consider those $R$ with $\ker R = 0$, and Lemma \ref{preA} demonstrates that this is equivalent to each $\varphi_i \in S^2(V)$ being nondegenerate.  Moreover, it is shown in \cite{G07} that $(V, R_{\varphi})$ is an indecomposable model space if and only if $\varphi$ is nondegenerate, so that our assumed decomposition is, in a certain sense, a complete one.

In this section we establish our main result regarding the structure group of the decomposable model space $(V, R) = \oplus_{i = 1}^k (V_i, R_{\varphi_i})$, with the assumption $\ker R = 0$.  We begin with a straightforward lemma.

\begin{lemma}  \label{preA}
Let $(V, R) = (V_1, R_1) \oplus (V_2, R_2)$ be a decomposable model space with $\ker R = 0$.  Let $\varphi_1, \varphi_2, \varphi \in S^2(V^*)$.
\begin{enumerate}
\item  We have $\ker R_i \cap V_i = 0$, for $i = 1, 2$.
\item  If $R_i = R_{\varphi_i}$, and ${\rm Rank}(\varphi_i) \geq 2$ for $i = 1, 2$, then $\ker \varphi_i \cap V_i = 0$.
\item  Suppose $\ker \varphi = 0$, and $\beta = \{e_1, \ldots, e_\ell\}$ is an orthonormal basis with respect to $\varphi$.  Then, for every $e_i \in \beta$ and $j \neq i$, the entries $R_\varphi(e_i, e_j, e_j, e_i)\neq 0$ are the only nonzero entries of $R_\varphi$ on this basis.  
\end{enumerate}
\end{lemma}
\begin{proof}
We establish Assertion (1) by proving $\ker R_i \cap V_i \subseteq \ker R = 0$.  Without loss of generality, suppose $v \in \ker R_1 \cap V_1$. Then since $v \in V_1$ and $R = R_1 \oplus R_2$, for any $x, y, z \in V$ we have $R(v, x, y, z) = R_1(v, \tilde x, \tilde y, \tilde z)$, where $\tilde x, \tilde y,$ and $\tilde z$ are the projections of $x, y,$ and $z$, respectively, to $V_1$.  Since $v \in \ker R_1$, we have $R(v, x, y, z) = R_1(v, \tilde x, \tilde y, \tilde z) = 0.$  Thus $v \in \ker R$, and so $v = 0$.

Assertions (2) and (3) are almost immediate.  Since $\ker R = 0$, by Assertion (1), we have $\ker R_{\varphi_i} \cap V_i= 0$.  Since $\ker R_{\varphi_i} = \ker \varphi_i$ when ${\rm Rank}(\varphi_i) \geq 2$, Assertion (2) follows.  Assertion (3) follows as well, since in the event $\ker \varphi = 0$, there are no null vectors in any orthonormal basis, and so 

$$R_\varphi(e_i, e_j, e_j, e_i) = \varphi(e_i, e_i) \varphi (e_j, e_j) = \pm 1 \neq 0.$$
In the event that $i, j, k$ are distinct, then one verifies $R_\varphi(e_i, e_j, e_k, e_\ell) = 0.$
\end{proof}

We may now prove Lemma \ref{technical}, and the proof of our main result Theorem \ref{multiphi} will follow as a direct result. We first  pause to establish some useful notation.  Suppose $V_i$ are subspaces of $V$ with $\dim(V) = N$, and  $\oplus_{i = 1}^k V_i = V$, and $\beta_i$ is a basis for $V_i$.  Now $\beta = \beta_1 \cup \cdots \cup \beta_k = \{e_1, \ldots, e_N\}$ is an ordered basis for $V$.    Set $a_1 = 1,$ and for each $r = 2, \ldots, k$, recursively set $a_r = a_{r-1} + \dim(V_{r-1})$, so that $e_{a_p}$ is the first basis vector of $\beta_p= \{e_{a_p}, \ldots, e_{a_{p+1}-1}\}$.  For each index $i = 1, \ldots, N$, we define $n_i \in \{1, \ldots, k\}$ to be the unique number so that $e_i \in \beta_{n_i}$.  For example, $n_{a_p} = p$.  This tool will allow us to locate each basis vector according to the subspace that contains it.

The proof of Lemma \ref{technical} is somewhat technical, and so we provide a nontechnical summary and short example to demonstrate the method of proof. 

\begin{lemma}  \label{technical}
Suppose $(V, R) = \oplus_{i = 1}^k (V_i, R_{\varphi_i})$ is a model space with $\ker R = 0$, and let $A \in G_R$.  Let $\beta_i$ be an orthonormal basis for $V_i$ with respect to $\varphi_i$, and create the ordered basis $\beta$ for $V$ as above.  Let $f_i = Ae_i = \sum_j a_{ji}e_j$.
\begin{enumerate}
\item   For each $i = 1, \ldots, N$, there exists a well-defined number $w_i$ so that if $n_i \neq n_t$, for any $s$ with $n_{w_i} = n_s$, then $a_{st} = 0$.
\item  For any $t$ with $n_t \neq n_i$, we have $f_t \in {\rm Span}(\cup_{j \neq n_{w_i}} \beta_j)$.
\item  The assignment $n_i \mapsto \sigma(i) = n_{w_i}$ is  well-defined.

\item  The function $\sigma (n_i) = n_{w_i}$ is a permutation.
\item  $Ae_i \in {\rm Span}(\beta_{\sigma(n_i)})$ for all $i \in \{1, \ldots, N\}$.
\item  $A: V_i \to V_{\sigma(i)}$ for all $i \in \{1, \ldots, k\}$.
\end{enumerate}
\end{lemma}

\begin{proof}  
Note that the hypotheses forbid $\dim V_i <2$, since otherwise $\ker R \neq 0$. Since $A$ is nonsingular, $A\beta = \{f_1, \ldots, f_N\}$ is a new basis for $V$.  Since $A \in G_R$, we must have 
$$
R(e_i, e_j, e_j, e_i) = A^*R(e_i, e_j, e_j, e_i) = R(f_i, f_j, f_j, f_i) = \varepsilon_{ij},
$$
where 
$$\varepsilon_{ij} = \left\{  \begin{array}{c l}
\pm1 & {\rm\ if\ } n_i = n_j \\
0 &  {\rm\ if\ } n_i \neq n_j
\end{array}
\right..$$
By Lemma \ref{preA}, for $s \neq w$ and $n_s = n_w$,
\begin{equation} \label{a}
R(f_r, e_s, f_t, e_w) = \pm a_{wr}a_{st},
\end{equation}
and if additionally, $n_r \neq n_t$, then by assumption $R(f_r, x, f_t, y) = A^*R(e_r, \tilde x, e_t, \tilde y) = 0$, where $\tilde x = A^{-1} x,$ and $\tilde y = A^{-1} y$.  So we conclude that if $s \neq w$, $n_s = n_w$, and $n_r \neq n_t$, then $R(f_r, e_s, f_t, e_w) = a_{wr}a_{st} = 0$.

We are now ready to establish Assertion (1).  Let $i\in \{1, \ldots, N\}$ be chosen.  Since $Ae_i = \sum a_{ji} e_j \neq 0$, there exists some smallest index $w$ so that $a_{wi} \neq 0$.  Set $w_i = w$; note that this assignment is well-defined since we choose the smallest such $w$.  Choose any index $t$ with $n_t \neq n_i$.  According to Equation (\ref{a}) we conclude that $R(f_i, e_s, f_t, e_w) = \pm a_{wi}a_{st} = 0$ so that $a_{st} = 0$ when $s \neq w$ and $n_s = n_w$.  We show presently that $a_{wt} = 0$ for all $t$ with $n_t \neq n_i$, completing the proof of Assertion (1).   

If there is a $t$ with $n_t \neq n_i$ with $a_{wt} \neq 0$, then for any $a_{cb}$ with $n_b = n_i$ and $c \neq w$, we have $R(f_t, e_c, f_b, e_w) =  \pm a_{wt}a_{cb} = 0.$  But now $a_{cb} = 0$ for all $b$, since we have already considered the case $n_b \neq n_i$ (for $b = t$ above).  It follows that $e_b \notin {\rm Span}\{f_1, \ldots, f_N\}$, a contradiction.

We now prove Assertion (2) as a straightforward consequence of Assertion (1).   Fix $i$, find $w_i$ according to Assertion (1), and choose any $t$ with $n_i \neq n_t$.  We have $f_t = Ae_t =  \sum_j a_{jt}e_j$.  By Assertion (1), for every $j$ with $n_j = n_{w_i}$, we have $a_{jt} = 0$, so that any vectors $e_j$ with $n_j = n_{w_i}$ do not appear in the sum $\sum_j a_{jt}e_j = f_t$.  Assertion (2) now follows.

We begin our proof of Assertion (3) with the following observation.   According to Assertion (1), we have that if $i \neq j$, then $n_{w_{a_i}} \neq n_{w_{a_j}}$, so that the function $i \mapsto n_{w_{a_i}}$ is injective, and hence bijective.  We presently show that $n_i \mapsto n_{w_i}$ is well-defined.

We suppose to the contrary that it is not.  This is equivalent to the assertion that there exists indices $q$ and $p$ with $n_{q} = p = n_{a_{p}}$, but $n_{w_q} \neq n_{w_{a_p}}$.  Since $i \mapsto n_{w_{a_i}}$ is surjective, there exists an index $r$ with $n_{w_{a_r}} = n_{w_q}$ and $a_r \neq a_p$.  So $n_{a_r} = r \neq p = n_{a_p} = n_q$, and Assertion (2) forces $f_q \in {\rm Span}(\cup_{j \neq n_{w_q}} \beta_j)$; this is a contradiction by the definition of $w_q$.

Assertion (4) follows easily, since one may freely choose to evaluate $\sigma$ using $a_p$ instead of any other $i$ with $n_i = p$, and thus $\sigma$ agrees with the bijection $n_i = p \mapsto n_{w_{a_p}} = n_{w_i}$.

We now prove Assertion (5).  Choose any index $i$.  By Assertion (4), for every $\ell$ with $n_\ell \neq n_{w_i}$ there exists an $a_p$ so that $n_{w_{a_p}} = n_\ell$, and $n_{a_p} \neq n_i$, and so Assertion (1) (using the index $a_p \in \{1, \ldots, N\}$) forces the coefficients $a_{i\ell}= 0$ for $n_\ell \neq n_{w_i}$, which shows $Ae_i \in {\rm Span}(\beta_{n_{w_i}})$.  Assertion (5) follows since, by definition, $n_{w_i} = \sigma(n_i)$.  Assertion (6) follows from Assertion (5).

 \end{proof}

Since the proof we give of Lemma \ref{technical} above is somewhat dense, we supply a representative example to illustrate our method of proof.

\begin{example}{\rm 
  Suppose that $k = 3$, and each $V_i$ is of dimension 2, so that $V = V_1 \oplus V_2 \oplus V_3$ has dimension 6.  Once the basis  $\beta$ is found as in the statement of Lemma \ref{technical}, we may express $A$ in terms of this basis, so that the $i$th column of $A$ are the components of the vector $f_i$ with respect to the basis $\beta$.  One divides the matrix up into pieces, according to the numbers $n_i$.  Since $f_1 = \sum_{j = 1}^6 a_{j1}e_j$, one of these $a_{j1} \neq 0$.  Suppose the smallest $w$ so that $a_{w1} \neq 0$ is when $w = 4$.  So we have $w_1 = 4$, so that $n_{w_1} = 2$.  This  means that $a_{11} = a_{21} = a_{31} = 0$, and $a_{41} \neq 0$.  The method of  proof of Assertion (1) is to first match the known nonzero entry $a_{41}$ with $a_{33}, a_{34}, a_{35},$ and $a_{36}$.  These are the $a_{st}$ with $n_t \neq n_1 = 1$, and $s$ with $n_s = n_{w_i} = 2$, but $s = 3 \neq 4 = w_1$.  Equation (\ref{a}) shows that these $a_{st} = 0$:
$$
A = \left[
\begin{array}{c c|c c|c c}
a_{11}&a_{12}&a_{13}&a_{14}&a_{15}&a_{16}\\
a_{21}&a_{22}&a_{23}&a_{24}&a_{25}&a_{26}\\ \hline
a_{31}&a_{32}&a_{33}&a_{34}&a_{35}&a_{36}\\
a_{41}&a_{42}&a_{43}&a_{44}&a_{45}&a_{46}\\ \hline
a_{51}&a_{52}&a_{53}&a_{54}&a_{55}&a_{56}\\
a_{61}&a_{62}&a_{63}&a_{64}&a_{65}&a_{66}\\ 
\end{array}
\right]  =
\left[
\begin{array}{c c|c c|c c}
0&a_{12}&a_{13}&a_{14}&a_{15}&a_{16}\\
0&a_{22}&a_{23}&a_{24}&a_{25}&a_{26}\\ \hline
0&a_{32}&0&0&0&0\\
a_{41}&a_{42}&a_{43}&a_{44}&a_{45}&a_{46}\\ \hline
a_{51}&a_{52}&a_{53}&a_{54}&a_{55}&a_{56}\\
a_{61}&a_{62}&a_{63}&a_{64}&a_{65}&a_{66}\\ 
\end{array}
\right] .
$$
If any of $a_{43}, a_{44}, a_{45}$ or $a_{46}$ are nonzero (now $s = 4 = w_1$), then we could conclude using Equation (\ref{a}) that $a_{32} = 0$ as well, in which case $A$ is singular: a contradiction.  So we must have 
$$A = 
\left[
\begin{array}{c c|c c|c c}
0&a_{12}&a_{13}&a_{14}&a_{15}&a_{16}\\
0&a_{22}&a_{23}&a_{24}&a_{25}&a_{26}\\ \hline
0&a_{32}&0&0&0&0\\
a_{41}&a_{42}&a_{43}&a_{44}&a_{45}&a_{46}\\ \hline
a_{51}&a_{52}&a_{53}&a_{54}&a_{55}&a_{56}\\
a_{61}&a_{62}&a_{63}&a_{64}&a_{65}&a_{66}\\ 
\end{array}
\right]
=
\left[
\begin{array}{c c|c c|c c}
0&a_{12}&a_{13}&a_{14}&a_{15}&a_{16}\\
0&a_{22}&a_{23}&a_{24}&a_{25}&a_{26}\\ \hline
0&a_{32}&0&0&0&0\\
a_{41}&a_{42}&0&0&0&0\\ \hline
a_{51}&a_{52}&a_{53}&a_{54}&a_{55}&a_{56}\\
a_{61}&a_{62}&a_{63}&a_{64}&a_{65}&a_{66}\\ 
\end{array}
\right].
$$
One can see now that $Ae_i$ is not in the span of $e_3$ and $e_4$ for $i = 3, 4, 5, 6$, the conclusion of Assertion (2).

To continue our example, we study $f_3 = Ae_3$.  Suppose $a_{13} = a_{23} = 0$, and $a_{53} \neq 0$, so that $w_3 = 5$, $n_{w_3} = 3$, and $Ae_3 = a_{53} e_5 + a_{63}e_6$.    Applying Assertion (1) to the lower left and lower right elements of $A$, we have 
$$
A = 
\left[
\begin{array}{c c|c c|c c}
0&a_{12}&a_{13}&a_{14}&a_{15}&a_{16}\\
0&a_{22}&a_{23}&a_{24}&a_{25}&a_{26}\\ \hline
0&a_{32}&0&0&0&0\\
a_{41}&a_{42}&0&0&0&0\\ \hline
a_{51}&a_{52}&a_{53}&a_{54}&a_{55}&a_{56}\\
a_{61}&a_{62}&a_{63}&a_{64}&a_{65}&a_{66}\\ 
\end{array}
\right] = 
\left[
\begin{array}{c c|c c|c c}
0&a_{12}&0&a_{14}&a_{15}&a_{16}\\
0&a_{22}&0&a_{24}&a_{25}&a_{26}\\ \hline
0&a_{32}&0&0&0&0\\
a_{41}&a_{42}&0&0&0&0\\ \hline
0&0&a_{53}&a_{54}&0&0\\
0&0&a_{63}&a_{64}&0&0\\ 
\end{array}
\right].
$$
One sees again that Assertion (2) holds, since $Ae_i$ is not in the span of $e_5$ and $e_6$ for $i = 1, 2, 5, 6$.

Finally, it is now clear that either $a_{15}$ or $a_{25}$ is nonzero, so that $n_{w_5} = 1$.  A final application of Assertion (1) shows that the upper left and upper center elements of $A$ are all zero:
$$A = 
\left[
\begin{array}{c c|c c|c c}
0&a_{12}&0&a_{14}&a_{15}&a_{16}\\
0&a_{22}&0&a_{24}&a_{25}&a_{26}\\ \hline
0&a_{32}&0&0&0&0\\
a_{41}&a_{42}&0&0&0&0\\ \hline
0&0&a_{53}&a_{54}&0&0\\
0&0&a_{63}&a_{64}&0&0\\ 
\end{array}
\right]
= 
\left[
\begin{array}{c c|c c|c c}
0&0&0&0&a_{15}&a_{16}\\
0&0&0&0&a_{25}&a_{26}\\ \hline
0&a_{32}&0&0&0&0\\
a_{41}&a_{42}&0&0&0&0\\ \hline
0&0&a_{53}&a_{54}&0&0\\
0&0&a_{63}&a_{64}&0&0\\ 
\end{array}
\right].
$$
Once one has done this analysis, the remaining assertions are easy to see.  The assignment $p \mapsto n_{w_{a_p}}$ is bijective, since every one of the basis vectors $e_1 = e_{a_1}, e_3 = e_{a_2},$ and $e_5 = e_{a_3}$ clearly must satisfy $\{1, 2, 3\} = \{n_{w_1}, n_{w_3}, n_{w_5}\}$.  That $n_{w_2} = n_{w_1}$ also follows from this fact:  if this were not the case, $n_{w_2}$ would be equal to $n_{w_3}$ or $n_{w_5}$, and as is evident in our matrix above (as a result of Assertion (1)), this is not possible.   Similarly for $n_{w_4} = n_{w_3}$, and $n_{w_6} = n_{w_5}$.  So $\sigma(n_i) = n_{w_i}$ is well-defined, and we may freely choose to evaluate $\sigma(n_i) = \sigma(n_{a_p})$ when $n_i = p = n_{a_p}$, and this function is bijective since $\{1, 2, 3\} = \{n_{w_1}, n_{w_3}, n_{w_5}\}$.  In fact, one can see from our example that $\sigma = (123)$, and $A:V_i \to V_{\sigma(i)}$.  \hfill $\qedbox$
}
\end{example}

Theorem \ref{multiphi} follows directly from Lemma \ref{technical}.  We remark in passing that our method of proof will carry over to any $\alpha \in \otimes^\ell V^*$ that satisfies the relations found in Assertion (3) of Lemma \ref{preA}.  It may be possible to apply this methodology to other contravariant tensors to obtain similar results.  Specifically, there is another construction of algebraic curvature tensors using antisymmetric 2-forms \cite{G01, G07}, and our method of proof of Lemma \ref{technical} and Theorem \ref{multiphi} apply equally well in that circumstance as well.

\section{Corollaries of Theorem \ref{multiphi}}  \label{section4}

This section is devoted entirely to establishing the corollaries of Theorem \ref{multiphi}.  We begin by establishing Corollary \ref{directproduct}.

\begin{proof}[Proof of Corollary \ref{directproduct}.]
Let $\beta_i$ be a basis for $V_i$, and let  $\beta = \beta_1 \cup \beta_2$ be an ordered basis for $V$.  Then the hypothesis that each $V_i$ is $g-$invariant for all $g \in G_R$ implies that, given any element $g \in G_R$, there exists matrices $g_1$ and $g_2$ so that 
\begin{equation}  \label{gstuff}
[g]_\beta = \left[
\begin{array}{c c}
g_1 & 0 \\
0 & g_2
\end{array}
\right].
\end{equation}
On this basis, let $j_i: G_{R_i} \to G_{R}$ be the inclusion of $G_{R_i}$ into $G_R$ defined as follows:
\begin{equation}  \label{inclusion}
j_1(g_1)=  \left[
\begin{array}{c c}
g_1 & 0 \\
0 & I
\end{array}
\right], \quad{\rm and} \quad j_2(g_2) = 
 \left[
\begin{array}{c c}
I & 0 \\
0 & g_2
\end{array}
\right]
\end{equation}
 It follows that $G_{R_i} \cong j_i(G_{R_i})$.  We complete the proof by showing that $G_R$ is the internal direct product of $j_1(G_{R_1})\cong G_{R_1}$ and $j_2(G_{R_2}) \cong G_{R_2}$.

First, we note that $j_1(G_{R_1}) \cap j_2(G_{R_2})$ is trivial according to Equations (\ref{gstuff}) and (\ref{inclusion}).  Also note that for every $g \in G_R$, we have 
$$
[g]_\beta = \left[
\begin{array}{c c}
g_1 & 0 \\
0 & g_2
\end{array}
\right] = j_1(g_1) j_2(g_2),
$$
so that $G_R = j_1(G_{R_1}) j_2(G_{R_2})$.  Finally, if $A \in j_1(G_{R_1})$ and $g \in G_R$, then expressing the conjugate of $A$ by $g$ as a matrix shows that $g^{-1}Ag \in G_R$ leaves $V_1$ invariant and is the identity on $V_2$.  Thus for any $x, y, z, w \in V_1$, we have
$$
\begin{array}  {r c l}
(g^{-1}Ag)^* R_1(x, y, z, w) & = & R_1((g^{-1}Ag)x, (g^{-1}Ag)y, (g^{-1}Ag)z, (g^{-1}Ag)w) \\
& = & R((g^{-1}Ag)x, (g^{-1}Ag)y, (g^{-1}Ag)z, (g^{-1}Ag)w) \\
& = & (g^{-1}Ag)^* R(x, y, z, w) \\
& = & R(x, y, z, w) \\
& = & R_1(x, y, z, w).
\end{array}
$$
So $g^{-1}Ag \in G_{R_1}$, and $G_{R_1} \unlhd G_R$.  Similarly for $G_{R_2}$.  We conclude $G_R = j_1(G_{R_1}) \times j_2(G_{R_2}) \cong G_{R_1} \times G_{R_2}$. 
\end{proof}

The proof of Corollary \ref{dimensions} follows mostly from basic linear algebraic observations:

\begin{proof}[Proof of Corollary \ref{dimensions}.]  To prove Assertion (1), suppose $\dim V_i \neq \dim V_j$, and there exists an $A \in G_R$ with $A: V_i \to V_j$.  Since $A^{-1} \in G_R$ as well and $A^{-1}:V_j \to V_i$, without loss of generality we may assume $\dim(V_i) > \dim(V_j)$ in which case $A|_{V_i}: V_i \to V_j$ must not have full rank:  a contradiction.  

We now prove Assertion (2).   Suppose $x_i, y_i, z_i, w_i \in V_i$.  Then $Ax_i, Ay_i, Az_i,$ and $Aw_i \in V_j$, and 
$$
\begin{array}{r c l}
R(x_i, y_i, z_i, w_i) & = & R_{\varphi_i}(x_i, y_i, z_i, w_i), \\
R(x_i, y_i, z_i, w_i) & = & A^* R(x_i, y_i, z_i, w_i) \\ 
& = & R(Ax_i, Ay_i, Az_i, Aw_i) \\
 & = & R_{\varphi_j}(Ax_i, Ay_i, Az_i, Aw_i) \\
 & = & R_{A^*\varphi_j}(x_i, y_i, z_i, w_i).
\end{array}
$$
If ${\rm Rank}(\varphi_i) \geq 3$, then $\varphi_i = \pm A^*\varphi_j$ (see \cite{G07}), and the result follows in that case.  According to our assumptions and Assertion (1),  if ${\rm Rank}(\varphi_i) = 2$ we know $\dim V_i = \dim V_j = {\rm Rank}(\varphi_j) = 2$.  If  $\{x, y\}$ is a basis for $V_i$, then the following equation must hold:
$$
R_{\varphi_i}(x, y, y, x) = R_{\varphi_j}(Ax, Ay, Ay, Ax) = (\det A)^2 R_{\varphi_j}(x, y, y, x).
$$
We note that in dimension 2, the signature of $\varphi_i$ is determined by the sign of $R_{\varphi_i}(x, y, y, x) = \varphi_i(x, x)\varphi_i(y, y) - \varphi_i(x, y)^2$.  Thus, $R_{\varphi_i}(x, y, y, x)$ has the same sign as $R_{\varphi_j}(x, y, y, x)$ if and only if the conclusion of  Assertion (2) holds.

To prove the final assertion,  we find bases $\beta_s = \{e^s_1, \ldots, e^s_{\dim(V_s)}\}$ for $V_s$ which are orthonormal with respect to $\varphi_s$.  We define $Be^s_t = e^s_t$ for $s \neq i, j, Be^i_t = e^j_t,$ and $Be^j_t = e^i_t$.  In addition, by denoting $B|_{V_i}$ simply as $B_i$, we have
$$
\begin{array}{r c l}
R_{\varphi_i}(e^i_{t_1}, e^i_{t_2}, e^i_{t_3}, e^i_{t_4}) & = & R(e^i_{t_1}, e^i_{t_2}, e^i_{t_3}, e^i_{t_4}) \\
& = & B^* R(e^i_{t_1}, e^i_{t_2}, e^i_{t_3}, e^i_{t_4}) \\
 & = & R(e^j_{t_1}, e^j_{t_2}, e^j_{t_3}, e^j_{t_4}) \\
 & = & R_{\varphi_j}(e^j_{t_1}, e^j_{t_2}, e^j_{t_3}, e^j_{t_4}) \\
 & = & B_i^*R_{\varphi_j}(e^i_{t_1}, e^i_{t_2}, e^i_{t_3}, e^i_{t_4}).
\end{array}
$$
Thus, $R_{\varphi_i} = B_i^* R_{\varphi_j}$ as desired.
\end{proof}

Before establishing Corollaries \ref{wreath} and \ref{morewreath}, we pause to review the wreath product as a group theoretic construction (see page 172 in \cite{R99}).  Let $G$ be any group.  The symmetric group $S_k$ acts on  $\Pi_{i = 1}^k G$  (the direct product of $k$ copies of $G$)  by permuting the indices:  $\theta(\sigma) (g_1, \ldots, g_k) = (g_{\sigma^{-1}(1)}, \ldots, g_{\sigma^{-1}(k)})$.  Thus, $\theta: S_k \to Aut(\Pi_{i = 1}^k G)$ is a homomorphism.  The \emph{wreath product $G \wr S_k$ of $G$ by $S_k$} is the semidirect product 
$$G \wr S_k := (\Pi_{i = 1}^k G) \rtimes_\theta S_k$$
and accordingly has the binary operation given by 
$$
(h_1, \ldots, h_k; \tau) (g_1, \ldots, g_k; \sigma) = (h_1g_{\tau^{-1}(1)}, \ldots, h_kg_{\tau^{-1}(k)}; \tau\sigma).
$$

\begin{proof}[Proof of Corollary \ref{wreath}.]  For simplicity, denote $G_{R_{\varphi}}$ as $G$.  Define $\Phi: (G \wr S_k) \to G_R$ on the vector $(v_1, \ldots, v_k) \in V_1 \oplus \cdots \oplus V_k$ as
$$
\Phi(g_1, \ldots, g_k; \sigma)(v_1, \ldots, v_k) = (g_1v_{\sigma^{-1}(1)}, \ldots, g_kv_{\sigma^{-1}(k)}).
$$
One observes that $\Phi(g_1, \ldots, g_k; \sigma) \in G_R$.  In addition, if $(g_1, \ldots, g_k; \sigma) \in \ker \Phi$, then we must have $g_iv_{\sigma^{-1}(i)} = v_i$ for all $v_i \in V_i$, so that $\sigma$ is the identity, and $g_i$ is the identity for all $i$.  So $\Phi$ is injective. If $A \in G_R$, then there exists a $\sigma \in S_k$ as in Theorem \ref{multiphi} with $A: V_i \to V_{\sigma(i)}$.  If we denote $A_i = A|_{V_i} \in G$ and set $g_i = A_{\sigma^{-1}(i)}$, then one has $\Phi(g_1, \ldots, g_k; \sigma) = A$.  We show $\Phi$ preserves the wreath product structure to complete the proof that $\Phi$ is an isomorphism:
$$
\Phi(h_1, \ldots, h_k; \tau)\Phi(g_1, \ldots, g_k;\sigma)(v_1, \ldots, v_k)
$$
$$
\begin{array}{r c l} & = & \Phi(h_1, \ldots, h_k; \tau)(g_1v_{\sigma^{-1}(1)}, \ldots, g_kv_{\sigma^{-1}(k)}) \\
& = & (h_1g_{\tau^{-1}(1)}v_{\sigma^{-1}\tau^{-1}(1)}, \ldots, h_kg_{\tau^{-1}(k)}v_{\sigma^{-1}\tau^{-1}(k)}) \\
& = & (h_1g_{\tau^{-1}(1)}v_{(\tau\sigma)^{-1}(1)}, \ldots, h_kg_{\tau^{-1}(k)}v_{(\tau\sigma)^{-1}(k)}) \\
& = & \Phi(h_1g_{\tau^{-1}(1)}, \ldots h_k g_{\tau^{-1}(k)}; \tau\sigma)(v_1, \ldots, v_k) \\
& = & \Phi[(h_1, \ldots, h_k;\tau)(g_1, \ldots, g_k;\sigma)](v_1, \ldots, v_k).
\end{array}$$
\end{proof}

\begin{proof}[Proof of Corollary \ref{morewreath}.]  Set $V_p = \oplus_{i = 1}^{k_p}W_p$.  Since $\dim(W_1) \neq \dim(W_2)$, we have that each $V_p$ are invariant.  Using Corollary \ref{directproduct}, we identify the structure group $G_R$ as the direct product of the structure groups $G_{R_1} \times G_{R_2}$.  Using Corollary \ref{wreath}, we identify each $G_{R_p} \cong G_{\varphi_p} \wr Sym_{k_p}$.  The result follows.
\end{proof}

\section{Invariants, and Applications to Curvature Homogeneous Manifolds}  \label{section5}

In this section, we prove a general result about the isometry invariants of the decomposible model space $\oplus_i (V_i, R_{\varphi_i}, A_i^1, \ldots, A_i^s)$ through an understanding of $G_{R_{\varphi_i}}$, and give two examples of how our results apply to curvature homogeneous manifolds.  The first application classifies the structure group of a curvature homogeneous Skew-Tsankov Riemannian manifold as being the same model considered in Corollary \ref{morewreath}, and the second application exhibits a realization of the model space studied in Corollary \ref{dimensions} as a curvature homogeneous manifold and uses the result of Corollary \ref{dimensions} and Theorem \ref{invariants} to construct an isometry invariant not of Weyl type.  We begin by putting the theory of invariants of model spaces in context.

Let $(M,g)$ be a pseudo-Riemannian manifold.  At each point, there is a model space $\mathfrak{M}_P := (T_PM, g|_P, R|_P, \nabla R|_P, \ldots, \nabla^kR|_P)$, where $\nabla^iR$ are the covariant derivatives of the Riemann curvature tensor $R$.  If $\psi$ is a local isometry of $M$, and $\psi(P) = Q$, then there is a model space isomorphism from $\mathfrak{M}_Q$ to $\mathfrak{M}_P$.  Thus, if $\alpha(x)$ is an invariant of $\mathfrak{M}_x$, it must be that $\alpha(P) = \alpha(Q)$.  In this way, if one can find a model space invariant $\alpha(P)$ at each $P \in M$, if $\alpha$ is nonconstant then $(M,g)$ is not locally homogeneous.  As mentioned in the introduction, the Weyl scalar invariants are enough invariants to determine if a Riemannian manifold is locally homogeneous \cite{PTV96}, but not always enough to determine local homogeneity in the pseudo-Riemannian setting.

\begin{proof}[Proof of Theorem \ref{invariants}.]  Let $F(\alpha_1, \ldots, \alpha_s)$ be a symmetric function of the $\alpha_i$.  Let $G$ be the structure group of the model $\oplus_{i = 1}^s(V_i, R_{\varphi_i})$, and let $G_i$ be the structure group of $(V_i, R_{\varphi_i})$.  Let $\beta_i$ be an arbitrary basis for $V_i$, and construct $\beta = \cup_{i = 1}^s \beta_i$ as a basis for $V$.  Use the basis $\beta$ to evaluate the function $F(\alpha_1, \ldots, \alpha_s)$, and denote this quantity as $F(\alpha_1, \ldots, \alpha_s)(\beta) = F(\alpha_1(\beta_1), \ldots, \alpha_s(\beta_s))$.  We must show that if $g\cdot \beta$ is the new basis formed by applying $g\in G$ to each entry of $\beta$, then $F(\alpha_1, \ldots, \alpha_s)(\beta) = F(\alpha_1, \ldots, \alpha_s)(g\cdot \beta)$.

Let $g \in G$ be arbitary.  There exists a permutation $\sigma\in S_k$ that is determined by $g$ as found in Theorem \ref{multiphi}.  According to Theorem \ref{multiphi}, $g|_{V_i}\cdot\beta_i$ is a basis for $V_j$, where $j = \sigma(i)$.  Thus, according to the last part of Assertion (3) of Corollary \ref{dimensions}, and since $\alpha_j$ is an invariant of the model $(V_j, R_{\varphi_j}, A_i^1, \ldots, A_i^s)$, $\alpha_j(g|_{V_i}\cdot \beta_i) = \alpha_i(\beta_i)$.  Thus, the evaluation of $F$ on the basis $g\cdot \beta$ is, since $F$ is symmetric in the inputs $\alpha_i$, 
$$\begin{array}{r c l}
F(\alpha_1, \ldots, \alpha_s)(g\cdot \beta) &=& F(\alpha_1(g|_{V_{\sigma^{-1}(1)}} \cdot \beta_{\sigma^{-1}(1)}), \ldots, \alpha_s(g|_{V_{\sigma^{-1}(s)}} \cdot \beta_{\sigma^{-1}(s)})) \\
&=& 
F(\alpha_{\sigma^{-1}(1)}(\beta_{\sigma^{-1}(1)}), \ldots, \alpha_{\sigma^{-1}(s)}(\beta_{\sigma^{-1}(s)})) \\
&=& F(\alpha_1, \ldots, \alpha_s)(\beta).
\end{array}
$$ 
Thus, $F$ is an invariant of $\mathfrak{M}$. 
\end{proof}

\subsection{Skew-Tsankov manifolds}

We use the following definition from \cite{BFG07}.

\begin{definition}  \label{skewtsankov}
Let $\mathfrak{M} = (V, \varphi, R)$ be a model space, where $\varphi\in S^2(V)$ is nondegenerate, and $R \in \mathcal{A}(V)$.  Let $\mathcal{R}$ be the skew-symmetric curvature operator characterized by the equality $R(x,y,z,w) = \varphi(\mathcal{R}(x,y)z, w)$.  The model space $\mathfrak{M}$ is \emph{skew-Tsankov} if $\mathcal{R}(x,y)\mathcal{R}(z,w) = \mathcal{R}(z,w) \mathcal{R}(x,y)$ for all $x, y, z, w \in V$.  A pseudo-Riemannian manifold $(M,g)$ is \emph{skew-Tsankov} if the model spaces $\mathfrak{M}_P = (T_PM, g|_P, R)$ are skew-Tsankov for all $P$.
\end{definition}
Thus, if $(M,g)$ is curvature homogeneous, the algebraic and geometric versions of skew-Tsankov coincide, in a sense.  Skew-Tsankov models $(V, \varphi, R)$ are completely classified in the event that $\varphi$ is positive (or negative) definite \cite{BFG07}, although the classification question remains open for more general signatures for $\varphi$.

\begin{theorem}[\cite{BFG07}] \label{tsankov}
 If $\mathfrak{M}$ is a Riemannian skew-Tsankov model, then $\mathfrak{M} = \oplus_{i=1}^s \mathfrak{M}_i \oplus \mathfrak{K}$, where $\mathfrak{K} = (U, \varphi_U, 0)$, and $\mathfrak{M_i} = (V_i, \varphi_i, R_i)$, where $\dim(V_i) = 2$ and $\varphi_U$, $\varphi_i$ are positive definite.
\end{theorem}

If $R$ is an algebraic curvature tensor on a two-dimensional vector space, then it is easy to see \cite{DG04} that $R = R_{\varphi}$ for some symmetric $\varphi$.  In addition, by replacing $\varphi$ with $-\varphi$, such a $\varphi$ can be chosen to have signature $(0,2)$ or $(1,1)$.  Thus, any skew-Tsankov model is necessarily of the type considered in Corollary \ref{morewreath}, with the added consideration from Section \ref{section2} that $\ker(R) = U$.  The following Theorem follows immediately from Theorem \ref{invariants}.

\begin{theorem} \label{skewmodel}
Let $\mathfrak{M}$ be a Riemannian Skew-Tsankov model.  According to Theorem \ref{tsankov}, $\mathfrak{M}$ decomposes into $\oplus_{i=1}^s\mathfrak{M}_i\oplus \mathfrak{K}$ as in Theorem \ref{tsankov}.  Let $\kappa_i$ be the sectional curvature of $\mathfrak{M}_i$, and let $F(\kappa_1, \ldots, \kappa_s)$ be any symmetric function of the $\kappa_i$.  Then $F$ is an invariant of the model $\mathfrak{M}$.
\end{theorem}

\begin{remark}  {\rm 
In Theorem \ref{skewmodel}, each of the nontrivial direct summands $\mathfrak{M}_i$ have dimension 2.  Thus, the sectional curvature completely determines $R_i$ on $V_i$.
}\end{remark}

\subsection{Curvature homogeneous manifolds}

We use Theorem \ref{invariants} to construct invariants of direct sums of model spaces of the form $(V, R_\varphi)$.  These models have appeared in the literature frequently \cite{D00, DG05, DGS04, GIZ02, GIZ03} to study questions of curvature homogeneity and the Osserman conjecture in the higher signature setting, and we aim to present a family of curvature homogeneous manifolds relating to the decomposable model space $\oplus_{i = 1}^s(V_i, R_{\varphi_i})$, and to construct new isometry invariants for direct sums of these models.

Since, in the Riemannian setting, the Weyl scalar invariants will determine if a curvature homogenous space is locally homogeneous (and determined up to isometry by these invariants \cite{PTV96}), an example in the Riemannian setting is not a place one would actually use the results in this paper for that purpose.  Instead, we construct a family of balanced signature manifolds as orthogonal products of those mentioned in the previous paragraph, all of whose Weyl scalar invariants will vanish.  We follow \cite{DG05} and \cite{G07} for our definition below.

\begin{definition}  \label{pp}
Let $(x_1, \ldots, x_p, y_1, \ldots, y_p)$ be coordinates on $M = \mathbb{R}^{2p}$.  Let the indices $i, j$ vary from $1$ to $p$.  Let $f(x_1, \ldots, x_p)$ be a smooth real valued function of the $x_i$, and define a smooth metric $g$ on $M$ as having the following nonzero components:
$$
g(\partial_{x_i}, \partial_{x_j}) =\frac{\partial f}{\partial x_i} \cdot\frac{\partial f}{\partial x_j}, \quad g(\partial_{x_i}, \partial_{y_i}) = 1.
$$
Denote $\mathcal{M}_f$ as the pseudo-Riemannian manifold $(M,g)$ built from $f$.  It has balanced signature $(p,p)$.
\end{definition}

Let $\mathfrak{M} = (V, (\cdot, \cdot), R_\varphi)$ be the model space with $$V = \mathbb{R}^{2p} = {\rm span}\{e_1, \ldots, e_p, f_1, \ldots, f_p\},$$ the metric has only the nonzero entries $(e_i, f_i) = (f_i, e_i) = 1$, and $\varphi \in S^2(V)$ with $\ker \varphi = \ker R_\varphi = {\rm span}\{f_1, \ldots, f_p\}$.  The following is known about this family, see \cite{D00, DG05, DGS04, G07} for details:

\begin{theorem}  \label{ppalreadyknown}
Let $\mathcal{M}_f$ and $\mathfrak{M}$ be defined as above, and let $p \geq 3$. Assume that the Hessian $H$ of $f$ has rank $p$ and has constant signature.  Let $\varphi \in S^2(V)$ have the same (constant) signature as $H$. 
\begin{enumerate}
\item  The curvature tensor $R$ and its covariant derivative $\nabla R$ of $\mathcal{M}_f$ satisfies 
\begin{enumerate}
\item  $R = R_H$.
\item  $\nabla R(Z_1, Z_2, Z_3, Z_4; Z_5) = Z_5(R(Z_1, Z_2, Z_3, Z_4))$.
\item  $\ker R = \ker \nabla R = {\rm span}\{\partial_{y_1}, \ldots, \partial_{y_p}\}.$
\end{enumerate}
\item  $\mathcal{M}_f$ is curvature homogeneous with model $\mathfrak{M}$.
\item  $\mathcal{M}_f$ is a generalized plane wave manifold.  Thus, all Weyl scalar invariants vanish, and $\mathcal{M}_f$ is complete.
\item  Suppose $H$ and $\varphi$ have signature $(r,s)$.  If $f = \frac12(-x_1^2 - \cdots - x_r^2 + x_{r+1}^2 + \cdots + x_{r+s}^2)$, then $\mathcal{M}_f$ satisfies $\nabla R = 0$; that is, $\mathcal{M}_f$ is symmetric and hence locally homogeneous.  Thus, $\mathfrak{M}$ is the model of a symmetric space.
\item  For generic choices of $f$, $\mathcal{M}_f$ is \emph{not} locally homogeneous.
\end{enumerate}
\end{theorem}

The following is a construction of an isometry invariant $\alpha_f$ of $\mathcal{M}_f$ that proves Assertion (5) of Theorem \ref{ppalreadyknown}; this invariant was originally constructed in \cite{DG05}.  We use the language and results of this paper to rephrase this construction in an effort to keep the paper self-contained.   Define the model space $\bar{\mathfrak{M}}_P = (\bar V,  \bar R, \bar{A})$, where the elements in this model space are as follows:

\begin{tabular}{c r c l}
&$\bar V$ & $=$ & $T_PM/(\ker R|_P)$,  with $\pi: T_PM \to \bar V$ the natural \\
 & & & projection. \\
$R|_P =$& $(R_H)|_P$ &$=$& $\pi^*\bar R$. \\
& $\nabla R|_P$ &$=$& $\pi^*\bar A$.
\end{tabular} 

\noindent One notices that $\bar R = R_{\bar \varphi}$, where $H = \pi^*\bar\varphi$, and according to Theorem \ref{Gphi}, the structure group of the model space $(\bar V, \bar R)$ is the group of linear transformations $A$ that have $A^* \bar\varphi = \pm \bar\varphi$ or $A^*\bar\varphi = \bar\varphi$, depending on the signature of $\bar\varphi$.  One then defines $\alpha_f$ as the absolute value of the square length of $\bar A$ with respect to $\bar\varphi$.  Specifically, if $\{X_1, \ldots, X_p\}$ is a basis for $\bar V$ that is orthonormal with respect to $\bar\varphi$, then 
$$
\alpha_f :=  $$  $$ \left|\sum_{ijk\ell n} \bar\varphi(X_i, X_i) \bar\varphi(X_j, X_j)\bar\varphi(X_k, X_k)\bar\varphi(X_\ell, X_\ell)\bar\varphi(X_n, X_n)\bar A(X_i, X_j, X_k, X_\ell; X_n)^2 \right|.
$$
Since any change of orthonormal basis with respect to $\bar\varphi$ preserves $\alpha_f$, it is invariant under the action of the structure group of $(\bar V, \bar R)$ (this is the purpose of the absolute value above, in the event that $A^* \bar\varphi = -\bar\varphi$), and hence an invariant of $\bar{\mathfrak{M}}_P$.  Since $\alpha_f$ is built from quantities that are preserved by isometry, $\alpha_f$ is an isometry invariant.  One may show that for generic $f$, the invariant $\alpha_f$ is not constant, and hence in this case $\mathcal{M}_f$ is not locally homogeneous.  In particular, $\alpha_f$ is an invariant of $\mathcal{M}_f$ which is \emph{not} of Weyl type, otherwise it would vanish according to Assertion (3) of Theorem \ref{ppalreadyknown}.  See \cite{DG05} or \cite{G07} for more on this invariant.  A similar invariant is constructed in the signature $(2,2)$ case, see or \cite{D00}, or \cite{DGS04} for somewhat different approaches to constructing this invariant--both rely on an understanding of the structure group of $(V, R_\varphi)$ for their construction.

We conclude this section and construct an example that realizes the model space in Corollary \ref{dimensions}, and illustrates the use of Theorem \ref{invariants} in the geometric setting.

\begin{theorem} \label{ppexample}
Define $\mathcal{M} = \times_{i = 1}^s \mathcal{M}_{f_i}$ as the orthogonal direct product of the manifolds $\mathcal{M}_{f_i}$.  Let $\alpha_{f_i}$ be the isometry invariant of $\mathcal{M}_i$ constructed above.  Then any symmetric function of the $\alpha_{f_i}$ is an isometry invariant of $\mathcal{M}$.  This isometry invariant is \emph{not} of Weyl type.
\end{theorem}
\begin{proof}
This theorem follows directly from the fact that the $\alpha_{f_i}$ are invariants, and by Theorem \ref{invariants}.  That $\mathcal{M}$ is constructed to be a generalized plane wave manifold, any Weyl-type invariant vanishes, and for generic $f_i$, the $\alpha_{f_i}$ do not \cite{DG05}.
\end{proof}
\begin{remark}  {\rm 
Provided each $f_i$ satisfies the conditions of Theorem \ref{ppalreadyknown}, it is immediate that the space $\mathcal{M}$ satisfies Assertions (2)--(5) of that Theorem, except in Assertion (2), where the model space considered is of the form found on Corollary \ref{dimensions}.  In this example there is, however, a nontrivial kernal of the curvature tensor.  }
\end{remark}
\begin{remark}  {\rm 
The construction given in Theorem \ref{ppexample} gives a new set of algebraic curvature tensors that are geometrically realizable as the curvature tensor of a non-homogeneous curvature homogeneous space; the search for these of interest since not every algebraic curvature tensor satisfies this (see \cite{KP94}).
 }  \end{remark}

\section{Summary of results}  \label{section6}

This paper aims to understand the structure group of the indecomposable models space of the form $(V, R_{\varphi})$, or of the decomposable model space $(V, R)$, where $R = \oplus_{i = 1}^k R_{\varphi_i}$.  We also aim to apply these results to a geometrical situation of use. According to the discussion preceding the proof of Theorem \ref{Gphi}, we answer these questions by considering the situation where $\ker R = 0$, which, for the situations considered here, amount to assuming that the symmetric forms involved are all nondegenerate.    

Section \ref{section2} computes the structure group $G_{R_{\varphi}}$ as $G_{\varphi}$ unless $\varphi$ has rank 2, or is of balanced signature.  Section \ref{section3} contains our main result:  to each element of the structure group $G_R$ for $R = \oplus_{i = 1}^k R_{\varphi_i}$, there is a permutation $\sigma$ with $A: V_i \to V_{\sigma(i)}$.  We apply these results to various situations of broad interest in Section \ref{section4} that help us to classify, up to group isomorphism, the structure group $G_R$ when $R\in \mathcal{A}(V)$ is the direct sum of canonical algebraic curvature tensors.  We show  that if there are ever two subspaces of $V$ which are invariant by the action of the structure group, then the structure group itself decomposes as an internal direct product.  In this case, the decomposition of the model space gives rise to a decomposition of the structure group.  Such a situation arises if, for example if $V = \oplus_{i = 1}^kV_i$, the subspaces $V_i$  have different dimensions.  In the event the subspaces $V_i$ have the same dimension and the forms $\varphi_i$ all have the same signature (or reversed signature), then the structure group can be recovered entirely from this data as the wreath product of $G_{R_{\varphi}}$ (which has been computed already in Section \ref{section2}), by the full symmetric group $S_k$.   We also show that in the event the subspaces $V_i$ and $V_j$ share the same dimension but $\varphi_i$ and $\varphi_j$ have incompatible signatures, then any element of the structure group must \emph{not} permute $V_i$ to $V_j$.  We close our study by noting that combinations of these results are also possible, and these combinations allow one to determine the (group) isomorphism class of $G_R$.  We finish our study by describing how one would apply these results in two geometrical settings.

\section*{Acknowledgments} The authors would like to thank B. Lim, Z. Hasan, S. Dunn, R. Trapp, J. Sutliff-Sanders, and S. Zwicknagl for helpful conversations while this research was conducted.  This research was jointly funded by the NSF grants DMS-0850959 and DMS-1156608, and California State University, San Bernardino.

\end{document}